\documentclass[reqno]{amsart}

\usepackage[utf8]{inputenc}
\usepackage{amsmath}
\numberwithin{equation}{section}
\numberwithin{figure}{section}
\numberwithin{table}{section}
\usepackage{graphicx}
\usepackage[mathscr]{euscript}
\usepackage{amsfonts}
\usepackage[colorlinks=true, allcolors=blue]{hyperref}
\usepackage{subfigure} 
\usepackage{color}
\usepackage{extarrows}
\usepackage{booktabs}
\usepackage{hyperref}
\usepackage{tikz}
\usepackage{extarrows}
\usepackage{booktabs}
\usepackage{amsthm}
\usepackage{enumerate}
\usepackage{enumitem} 
\usepackage{indentfirst}
\usepackage{embrac}
\usepackage{cite} 
\usetikzlibrary{positioning,shapes}

\def\sE{{\mathscr E}}
\def\sF{{\mathscr F}}

\def\${|\!|\!|}
\def\l|{\left|\!\left|\!\left|}
\def\r|{\right|\!\right|\!\right|}

\newtheorem{theorem}{Theorem}[section]
\newtheorem{lemma}[theorem]{Lemma}
\newtheorem{proposition}[theorem]{Proposition}
\newtheorem{corollary}[theorem]{Corollary}
\theoremstyle{definition}
\newtheorem{definition}[theorem]{Definition}

\theoremstyle{remark}
\newtheorem{remark}[theorem]{Remark}
\numberwithin{equation}{section}

\begin{document}

\title[Order isomorphisms]{On order isomorphisms intertwining semigroups for Dirichlet forms}

\author{Liping Li}
\address{Fudan University, Shanghai, China.  }
\address{Bielefeld University,  Bielefeld, Germany.}
\email{liliping@amss.ac.cn}
\thanks{The first named author is partially supported by NSFC (No.  11931004) and Alexander von Humboldt Foundation in Germany.}

\author{Hanlai Lin}
\address{Fudan University, Shanghai, China.  }
\email{}





\begin{abstract}
This paper is devoted to characterizing the so-called order isomorphisms intertwining the $L^2$-semigroups of two Dirichlet forms.   We first show that every unitary order isomorphism intertwining semigroups is the composition of $h$-transformation and quasi-homeomorphism.  In addition,  under the absolute continuity condition on Dirichlet forms,  every (not necessarily unitary) order isomorphism intertwining semigroups is the composition of $h$-transformation, quasi-homeomorphism, and multiplication by a certain step function.   
\end{abstract}

\maketitle
\tableofcontents

\section{Introduction}

In a famous paper \cite{K66},  Kac asked the following question: Let $\Omega_i\subset \mathbb{R}^d$,  $i=1,2$, be bounded domains satisfying certain regularity condition,  and $\Delta_i$ be the Laplacian with Dirichlet boundary condition on $\Omega_i$.  Does it follow that $\Omega_1$ and $\Omega_2$ are congruent if knowing that $\Delta_1$ and $\Delta_2$ have the same series of eigenvalues? This question can be formulated in terms of the $L^2$-semigroups $(T^i_t)_{t\geq 0}$ generated by $\Delta_i$.  Take an orthogonal basis $\{e^i_n: n\geq 1\}$ of $L^2(\Omega_i)$ consisting of eigenfunctions of $\Delta_i$ and let $U:L^2(\Omega_1)\rightarrow L^2(\Omega_2)$ be the linear operator such that $Ue^1_n=e^2_n$ for $n\geq 1$.  Then the assumption the the spectra coincide is equivalent to that
\begin{equation}\label{eq:11}
	UT^1_t=T^2_t U,\quad t\geq 0.  
\end{equation}
Kac's problem amounts to asking that does it follow that $\Omega_1$ and $\Omega_2$ are congruent if \eqref{eq:11} holds for a certain unitary operator $U:L^2(\Omega_1)\rightarrow L^2(\Omega_2)$? It has been shown in \cite{GWW92} that,  in general,  the answer is negative. 

In \cite{A02},  Arendt studied Kac's problem under an additional condition that $U$ is a so-called \emph{order isomorphism}: $U$ is bijective, and for $f\in L^2(\Omega_1)$,  $f\geq 0$ if and only if $Uf\geq 0$.  The main result of \cite{A02} stated that if $U$ is an order isomorphism satisfying \eqref{eq:11},  then $\Omega_1$ and $\Omega_2$ are congruent.  In a recent paper \cite{LSW18},  Lenz et al.  investigated an analogical problem in terms of Dirichlet forms.  The terminologies and notations concerning Dirichlet forms are referred to \cite{FOT11,  CF12}; see also \S\ref{SEC21}.  Let $(\sE^i,\sF^i)$ be two quasi-regular and irreducible Dirichlet forms on 
$L^2(E_i,m_i)$,  whose $L^2$-semigroups are denoted by $(T^i_t)_{t\geq 0}$, for $i=1,2$.  Consider an order isomorphism $U: L^2(E_1,m_1)\rightarrow L^2(E_2,m_2)$,  which is defined by the same way as the Laplacian case,  and assume that $U$ \emph{intertwines} $(T^1_t)_{t\geq 0}$ and $(T^2_t)_{t\geq 0}$ in the sense of \eqref{eq:11}.  Then Lenz et al. showed that the topologies of $E_1$ and $E_2$ coincide in the following sense: There exist $\sE^i$-nests $\{F^i_n:n\geq 1\}$ and a map $j: E_1\rightarrow E_2$ such that $j|_{F^1_n}$ is a homeomorphism from $F^1_n$ to $F^2_n$ for any $n\geq 1$.  This problem is also considered in \cite{DW21} for possibly reducible Dirichlet forms and unitary order isomorphisms intertwining their $L^2$-semigroups by an argument involving the so-called ergodic decompositions of Dirichlet forms. 

The main purpose of the current paper is to characterize the order isomorphisms intertwining the $L^2$-semigroups of two Dirichlet forms by means of transformations of Dirichlet forms.  We first note that two transformations of Dirichlet forms,  i.e.  \emph{$h$-transformation} and \emph{quasi-homeomorphism},  are corresponding to particular unitary order isomorphisms intertwining semigroups,  as will be explained in \S\ref{SEC22} and \S\ref{SEC23}.  Then it turns out in Theorem~\ref{THM31} that every unitary order isomorphism intertwining semigroups is the composition of $h$-transformation and quasi-homeomorphism.  Regarding general (not necessarily unitary) order isomorphisms intertwining semigroups,  we assume the absolute continuity condition for the probability transition semigroup of the Markov process associated with $(\sE^1,\sF^1)$ and obtain an irreducible decomposition of $E_1$ for $(\sE^1,\sF^1)$ in the sense of Kuwae \cite{K21}.  Then $E_2$ also admits an irreducible decomposition for $(\sE^2,\sF^2)$.  Our main result,  Theorem~\ref{THM45},  concludes that every order isomorphism intertwining semigroups is the composition of three transformations: $h$-transformation,  quasi-homeomorphism,  and multiplication by a step function which is constant on each invariant set of $(\sE^2,\sF^2)$.

The paper is organized as follows.  In \S\ref{SEC2} we prepare some terminologies and notations about Dirichlet forms and order isomorphisms.  Particularly,  Lemmas~\ref{LM22} and \ref{LM23} show that $h$-transformation and quasi-homeomorphism give special unitary order isomorphisms intertwining semigroups.  The characterization of unitary order isomorphisms intertwining semigroups will be completed in  \S\ref{SEC3}.  Then the general case will be studied in \S\ref{SEC4}. 


\section{$h$-transformation,  quasi-homeomorphisms and order isomorphisms}\label{SEC2}

In this section we prepare some terminologies and notations about Dirichlet forms and order isomorphisms.  

\subsection{Dirichlet forms}\label{SEC21}

Let $(E,\mathcal{B}(E))$ be a measurable space and $m$ be a $\sigma$-finite measure on it.  The completion of $\mathcal{B}(E)$ with respect to $m$ is denoted by $\mathcal{B}^m(E)$.  A Dirichlet form $(\sE,\sF)$ is a symmetric Markovian closed form on $L^2(E,m)$,  for which we refer to the standard textbooks \cite{CF12, FOT11}.  Set $\sE_1(f,g):=\sE(f,g)+(f,g)_{L^2(E,m)}$ for any $f,g\in\sF$.  

Assume further that $E$ is a Hausdorff topological space with the Borel $\sigma$-algebra $\mathcal{B}(E)$ being assumed to be generated by the continuous functions on $E$ and that $m$ is fully supported on $E$. 
An ascending sequence $\{F_n: n\geq 1\}$ of closed subsets of $E$ is called an \emph{$\sE$-nest} if $\cup_{n\geq 1} \sF_{F_n}$ is $\sE_1$-dense in $\sF$, where $\sF_{F_n}:=\{f\in\sF:f=0, m\text{-a.e.  on  } F_n^c\}$.   A subset $N$ of $E$ is called \emph{$\sE$-polar} if there is an $\sE$-nest $\{F_n:n\geq 1\}$ such that $N\subset \cap_{n\geq 1}(E\setminus F_n)$.  A statement depending on $x\in A$ is said to hold \emph{$\sE$-quasi-everywhere} (\emph{$\sE$-q.e.} in abbreviation) on $A$ if there is an $\sE$-polar set $N\subset A$ such that the statement is true for every $x\in A\setminus N$.  A function $f$ on $E$ is said to be \emph{$\sE$-quasi-continuous} (\emph{$\sE$-q.c.} in abbreviation) if there is an $\sE$-nest $\{F_n: n\geq 1\}$ such that $f|_{F_n}$ is finite and continuous on $F_n$ for each $n\ge 1$.

The Dirichlet form $(\sE,\sF)$ is called \emph{quasi-regular},  if 
  \begin{itemize}
      \item[(i)]  There exists an $\sE$-nest $\{F_n:n\geq 1\}$ consisting of compact sets;
      \item[(ii)]  There exists an $\sE_1$-dense subset of $\sF$ whose elements have $\sE$-q.c.  $m$-versions;
      \item[(iii)]  There exists $\{f_n: n\geq 1\}\subset \sF$ having $\sE$-quasi-continuous $m$-versions $\{\tilde{f}_n:n\geq 1\}$ and an $\sE$-polar set $N\subset E$ such that $\{\tilde{f}_n: n\geq 1\}$ separates the points of $E\setminus N$.
  \end{itemize}
It is  called \emph{regular},  if $E$ is a locally compact separable metric space,  $m$ is a Radon measure on $E$ with full support,  and $\sF\cap C_c(E)$ is $\sE_1$-dense in $\sF$ and uniformly dense in $C_c(E)$ respectively,  where $C_c(E)$ is the family of all continuous functions with compact support on $E$. 
  Note that a regular Dirichlet form is always quasi-regular.  
  
\subsection{$h$-transformation}\label{SEC22}

Let $(\sE,\sF)$ be a quasi-regular Dirichlet form on $L^2(E,m)$.  
Denote by $(T_t)_{t\geq 0}$ the $L^2$-semigroup of $(\sE,\sF)$.  Note that $T_t|_{L^2_+(E,m)}$ can be extended to an operator,  still denoted by $T_t$, on $L_+(E,m)$,  where $L^2_+(E,m)=\{f\in L^2(E,m): f\geq 0\}$ and $L_+(E,m)$ is the family of all $m$-equivalence classes of $\mathcal{B}^m(E)$-measurable functions from $E$ to $[0,\infty]$.    More precisely,  take a strictly positive function $\eta\in L^2(E,m)$ and define for any $f\in L_+(E,m)$, 
\[
T_t f:=\lim_{n \rightarrow \infty}T_t(f\wedge  n\eta).  
\]
Since $T_tg\geq 0$ for any $g\in L^2(E,m)$,  it follows that $T_t(f\wedge  n\eta)\uparrow T_t f$ and $T_t f\in L_+(E,m)$.  This definition is independent of the choice of $\eta$.  In fact,  set $f_n:=f\wedge n\eta$.  For another sequence $g_n\in L^2_+(E,m)\uparrow f$,  $\tilde{T}_t f:=\lim_{n\rightarrow \infty}T_tg_n$ is also well defined and $\tilde{T}_t f=T_t f$ can be obtained as follows: For any $h\in L^2_+(E,m)$,  we have
\[
\begin{aligned}
	\int_E h\cdot T_t f dm&=\lim_{n\rightarrow \infty}\int_E h\cdot T_tf_n dm=\lim_{n\rightarrow \infty}\int_E T_th \cdot f_n dm \\
	&=\int_E T_th\cdot fdm=\lim_{n\rightarrow \infty}\int_E T_th  \cdot g_n dm =\int_E h \cdot \tilde{T}_t f dm.
\end{aligned}\]
As a result, $T_t f=\tilde{T}_t f$.

 We introduce a family of excessive functions with respect to $(T_t)_{t\geq 0}$ as follows:
 \begin{equation}\label{eq:23}
\begin{aligned}
    \mathbf{Exc}^+:=\big\{&h\in L_+(E,m): T_t  h \le h,\forall t\geq 0 \text{ and }\\
    & h \text{ admits an }\sE\text{-q.c. } m\text{-version } \tilde{h} \text{ with } \tilde{h}>0, \sE\text{-q.e.}\big\}.  
\end{aligned}
\end{equation}
This family is  not necessarily contained in $L^2(E,m)$ (cf.,  e.g.,  \cite{FOT11, CF12}),  and constant functions belong to $\mathbf{Exc}^+$.  For $h\in \mathbf{Exc}^+$,  
\[
	E_h:=\{x\in E: 0<h(x)<\infty\}=E,\quad m\text{-a.e.},
\]
and
\begin{equation}\label{eq:22}
    T^{h}_tf(x):=\frac{1}{h(x)}T_t(hf)(x),\quad t\geq 0,  f\in L^2(E,h^2\cdot m)
\end{equation}
is called the \emph{$h$-transformed semigroup} on $L^2(E,h^2\cdot m)$.  Clearly $(T^h_t)_{t\geq 0}$ is a strongly continuous contraction Markovian semigroup on $L^2(E,h^2\cdot m)$.  

\begin{lemma}
Let $(\sE,\sF)$ be a quasi-regular Dirichlet form on $L^2(E,m)$ and $h\in \mathbf{Exc}^+$.  Then the $h$-transformed semigroup $(T^h_t)_{t\geq 0}$ is associated with the quasi-regular Dirichlet form on $L^2(E,h^2\cdot m)$:
\begin{equation}\label{eq:21}
    \begin{aligned}
         &\sF^{h}=\{f\in L^2(E,h^2\cdot m),fh \in \sF\}\\
         &\sE^{h}(f,g)=\sE(fh,gh),\quad f,g\in \sF^h.  
    \end{aligned}
\end{equation}
Furthermore,  the following hold:
\begin{itemize}
\item[(1)] An ascending sequence $\{F_n\}$ of closed subsets of $E$ is an $\sE^h$-nest,  if and only if it is an $\sE$-nest. 
\item[(2)] A set $N\subset E$ is $\sE^h$-polar if and only if it is $\sE$-polar.
\item[(3)] $f$ is $\sE^h$-q.c.  if and only if it is $\sE$-q.c. 
\end{itemize}
\end{lemma}
\begin{proof}
Clearly $h^2\cdot m$ is $\sigma$-finite and has full support in $E$,  and $(\sE^h,\sF^h)$ is a Dirichlet form on $L^2(E,h^2\cdot m)$.  It is easy to verify that $\cup_{n\geq 1}\sF^h_{F_n}$ is $\sE^h_1$-dense in $\sF^h$,  if and only if $\cup_{n\geq 1}\sF_{F_n}$ is $\sE_1$-dense in $\sF$.  Hence $\{F_n\}$ is an $\sE^h$-nest if and only if it is an $\sE$-nest.  Particularly,  $N$ is $\sE^h$-polar if and only if it is $\sE$-polar; every $\sE$-q.c. function is $\sE^h$-q.c.  and vice verse.  In addition, the first condition (i) in the definition of quasi-regularity holds true for $(\sE^h,\sF^h)$.  For another two conditions (ii) and (iii),  it suffices to note that for any $\sE$-q.c.  continuous $\tilde f\in \sF$,  $\tilde{f} \cdot \tilde{h}\in \sF^h$ is $\sE^h$-q.c.  That completes the proof. 
\end{proof}

The Dirichlet form $(\sE^h,\sF^h)$ is called the \emph{$h$-transform} of $(\sE,\sF)$.  
Define an isometry 
\begin{equation}\label{eq:25}
U_h: L^2(E,m)\rightarrow L^2(E, h^2\cdot m), \quad f\mapsto \frac{f}{h}
\end{equation}
 for $h\in \mathbf{Exc}^+$.  Then $(\sE^h,\sF^h)$ is the \emph{image Dirichlet form} of $(\sE,\sF)$ under $U_h$ in the sense that 
\[
	\sF^h=U_h \sF,\quad \sE^h(f,g)=\sE(U_h^{-1}f,U_h^{-1}g),\quad f,g\in \sF^h,
\]
where $U^{-1}_h$ is the inverse of $U_h$.  Note that \eqref{eq:22} yields the following.

\begin{lemma}\label{LM22}
Let $h\in \mathbf{Exc}^+$.  Then $U_h T_t f=T^h_t U_hf$ for $t\geq 0$ and $f\in L^2(E,m)$.  
\end{lemma}
  
\subsection{Quasi-homeomorphism of Dirichlet spaces}\label{SEC23}

Given another Dirichlet form $(\hat{\sE},\hat{\sF})$ on a second $L^2$-space $L^2(\hat{E},\hat{m})$, where $\hat{E}$ is also a Hausdorff topological space and $\hat{m}$ is fully supported on $\hat{E}$,  $(\sE,\sF)$ is said to be quasi-homeomorphic to $(\hat\sE,\hat\sF)$ if there is an $\sE$-nest $\{F_n: n\geq 1\}$,  an $\hat{\sE}$-nest $\{\hat{F}_n: n\geq 1\}$ and a map $j: \cup_{n\ge1} F_n\to \cup_{n\ge1} \hat F_n$ such that 
  \begin{itemize}
      \item[(i)] $j$ is a topological homeomorphism from $F_n$ onto $\hat F_n$ for each $n\ge 1$.
      \item[(ii)] $\hat m=m\circ j^{-1}$.
      \item[(iii)]  It holds that
      \[
      \begin{aligned}
     & \hat{\sF}=\{\hat{f}\in L^2(\hat{E},\hat{m}): \hat{f}\circ j\in \sF\},\\
      &	\hat{\sE}(\hat{f},\hat{g})=\sE(\hat{f}\circ j,\hat{g}\circ j),\quad \hat{f},\hat{g}\in \hat{\sF}. 
     \end{aligned} \]
  \end{itemize}
  The map $j$ is called a \emph{quasi-homeomorphism} from $(\sE,\sF)$ to $(\hat{\sE}, \hat{\sF})$.  
  Define an isometry 
  \begin{equation}\label{eq:26}
  U_j: L^2(E,m)\rightarrow L^2(\hat{E},\hat{m}), \quad  f\mapsto f\circ j^{-1}
  \end{equation}
for a quasi-homeomorphism $j$.  Then $(\hat{\sE},\hat{\sF})$ is the image Dirichlet form of $(\sE,\sF)$ under $U_j$,  i.e. 
  \[
  	\hat{\sF}=U_j\sF,\quad \hat{\sE}(f,g)=\sE(U^{-1}_j f,U^{-1}_j g),\quad f,g\in \hat{\sF},
  \]
 where $U_j^{-1}$ is the inverse of $U_j$.  It is worth pointing out that a quasi-regular Dirichlet form is always quasi-homeomorphic to a certain regular Dirichlet form,  and the quasi-homeomorphic image of a quasi-regular Dirichlet form is still quasi-regular. 
 
 Denote by $(T_t)_{t\geq 0}$ and $(\hat{T}_t)_{t\geq 0}$ the $L^2$-semigroup of $(\sE,\sF)$ and $(\hat{\sE},\hat{\sF})$ respectively.  The following lemma is obvious and the proof is omitted. 
 
 \begin{lemma}\label{LM23}
 Let $j$ be a quasi-homeomorphism from $(\sE,\sF)$ to $(\hat{\sE}, \hat{\sF})$.  Then $\hat{T}_tU_jf=U_jT_tf$ for any $t\geq 0$ and $f\in L^2(E,m)$.  
 \end{lemma}
  
\subsection{Order isomorphisms}\label{SEC24}

Let $(E_i,\mathcal{B}(E_i))$ be a measurable space and $m_i$ be a $\sigma$-finite measure on it for $i=1,2$.  The following concepts play a central role in this paper.  Note that we do not impose any topological structure on $E_i$ for these concepts. 

\begin{definition}
\begin{itemize}
\item[(1)] A linear map $U:L^2(E_1,m_1)\rightarrow L^2(E_2,m_2)$ is called \emph{positivity preserving} if $Uf\ge 0$, $m_2$-a.e., for $f\in L^2_+(E_1,m_1)$.
\item[(2)]  A positivity preserving map $U$ is said to be an \emph{order isomorphism} if $U$ has a positivity preserving inverse.  
\end{itemize}
\end{definition}
\begin{remark}
A positivity preserving operator is always bounded; see,  e.g.,  \cite[Theorem~4.3]{Aliprantis:2006jm}.  An order isomorphism $U$ is an invertible operator $U: L^2(E_1,m_1)\rightarrow L^2(E_2,m_2)$ such that for any $f\in L^2(E_1,m_1)$,  $Uf\geq 0$ if and only if $f\geq 0$.  Mimicking the argument concerning $T_t$ in the beginning of \S\ref{SEC22},  we can also extend $U|_{L^2_+(E_1,m_1)}$ to an operator from $L_+(E_1,m_1)$ to $L_+(E_2,m_2)$,  which is still denoted by $U$.  
\end{remark}

The following proposition presents a representation for order isomorphisms,  which is crucial to our treatment,  and the proof is referred to \cite[Proposition~3.2]{DW21}. 

\begin{proposition}\label{PRO26}
Assume that $E_i$ is a standard Borel space,  i.e.  isomorphic to a Polish space with its Borel $\sigma$-algebra, for $i=1,2$.  
Let $U:L^2(E_1,m_1) \to L^2(E_2,m_2)$ be an order isomorphism.  Then there exists a measurable map $s:(E_1,\mathcal{B}^{m_1}(E_1))\to (0,\infty)$ and a measurable map $\tau:(E_1,\mathcal{B}^{m_1}(E_1))\to (E_2,\mathcal{B}^{m_2}(E_2))$ with measurable a.e.  inverse $\tau^{-1}$ such that
\begin{equation}\label{eq:24}
    Uf=\left(\frac{f}{s}\right)\circ\tau^{-1}, \quad \forall f \in L^2(E_1,m_1).
\end{equation}
Here $s$ and $\tau$ are unique up to equality almost everywhere. 
\end{proposition}
\begin{remark}
The map $\tau$ is an \emph{almost-isomorphism} from $E_1$ to $E_2$ in the sense that there exist $m_i$-negligible sets $N_i\in \mathcal{B}^{m_i}(E_i)$ such that $\tau: E_1\setminus N_1\rightarrow E_2\setminus N_2$ is a strict isomorphism.  
\end{remark}

The $s$ and $\tau$ are called \emph{scaling} and \emph{transformation} associated with the order isomorphism $U$ respectively.  Clearly,  the inverse of $U$
\[
	U^{-1}g=s \cdot g\circ \tau,\quad g\in L^2(E_2,m_2)
\]
is also an order isomorphism.  
The scaling and transformation associated with $U^{-1}$ are $\frac{1}{s}\circ \tau^{-1}$ and $\tau^{-1}$.  

\section{Unitary order isomorphisms intertwining semigroups}\label{SEC3}

Let $(\sE^i,\sF^i)$ be a Dirichlet form on $L^2(E_i,m_i)$,  where $L^2(E_i,m_i)$ is given in \S\ref{SEC24},  for $i=1,2$.  Denote by $(T^i_t)_{t\geq 0}$ the $L^2$-semigroup of $(\sE^i,\sF^i)$.  An order isomorphism $U:L^2(E_1,m_1)\to L^2(E_2,m_2)$ is said to \emph{intertwine} $(T^1_t)_{t\ge0}$ and $(T^2_t)_{t\ge0}$,  if 
\begin{equation}\label{eq:31}
UT^1_tf=T^2_tUf,\quad \forall f\in L^2(E_1,m_1).
\end{equation}
 Denote by $\mathscr O$ the family of all order isomorphisms intertwining $(T^1_t)_{t\geq 0}$ and $(T^2_t)_{t\geq 0}$.  Set
\[
\mathscr U:=\left\{U\in \mathscr O: (Uf,Ug)_{L^2(E_2,m_2)}=(f,g)_{L^2(E_1,m_1)},\forall f,g \in L^2(E_1,m_1)\right\},
\]
i.e.  the family of all unitary order isomorphisms intertwining $(T^1_t)_{t\geq 0}$ and $(T^2_t)_{t\geq 0}$.  Any $U\in \mathscr U$ is an isometry between $L^2(E_1,m_1)$ and $L^2(E_2,m_2)$ and it is easy to verify that
\begin{equation}\label{eq:32}
	\sF^2=U\sF^1,\quad \sE^2(f,g)=\sE^1(U^{-1}f,U^{-1}g),\quad f,g\in \sF^2;
\end{equation}
in other words,  $(\sE^2,\sF^2)$ is the image Dirichlet form of $(\sE^1,\sF^1)$ under $U$. 

We endow $E_i$ with a Polish topological structure,  i.e.  $E_i$ is assumed to be a Polish space and $\mathcal{B}(E_i)$ is the Borel $\sigma$-algebra generated by the continuous functions on $E_i$,  for $i=1,2$.  Further assume that $(\sE^i,\sF^i)$ is quasi-regular.  It is worth emphasising that these assumptions are not necessary for the set-up of $U\in \mathscr O$.  Denote by $\mathbf{Exc}^+_1$ the family of excessive functions defined as \eqref{eq:23} with respect to $(T^1_t)_{t\geq 0}$ and by $(\sE^{1,h},\sF^{1,h})$ the $h$-transform of $(\sE^1,\sF^1)$ for $h\in \mathbf{Exc}^+_1$.  The $\sE^1$-q.c.  $m_1$-version of $h\in \mathbf{Exc}^+_1$ is denoted by $\tilde{h}$.  Note that both $U_h$ and $U_j$ corresponding to $h$-transformation and quasi-homeomorphism are unitary order isomorphisms intertwining semigroups; see Lemmas~\ref{LM22} and \ref{LM23}.    In general we have the following characterization of unitary order isomorphisms intertwining $(T^1_t)_{t\geq 0}$ and $(T^2_t)_{t\geq 0}$.   

\begin{theorem}\label{THM31}
Assume that $E_i$ is a Polish topological space with the Borel $\sigma$-algebra $\mathcal{B}(E_i)$ and $(\sE^i,\sF^i)$ is a quasi-regular Dirichlet form on $L^2(E_i,m_i)$ for $i=1,2$.  Then $U\in \mathscr U$,  if and only if there exists $h\in \mathbf{Exc}^+_1$ and a quasi-homeomorphism $j$ from $(\sE^{1,h},\sF^{1,h})$ to $(\sE^2,\sF^2)$ such that $U=U_jU_h$.  The pair $(\tilde h,j)$ for $U\in \mathscr U$ is unique up to $\sE^1$-q.e.  equality.
\end{theorem}
\begin{proof}
\emph{Sufficiency}.  Let $(h,j)$ be such a pair.  Clearly $U_jU_h$ is invertible because so are $U_j$ and $U_h$.  In addition,  $U_jU_hf\geq 0$ if and only if $f\geq 0$ for any $f\in L^2(E_1,m_1)$.  Hence $U_jU_h$ is an order isomorphism.  For the intertwining property,  it follows from Lemmas~\ref{LM22} and \ref{LM23} that for any $t>0$ and $f\in L^2(E_1,m_1)$,
\[
\begin{aligned}
	&T^2_t\left(U_jU_h\right)f=\left(T^2_t U_j \right)U_hf=\left(U_jT^{1,h}_t\right)U_hf\\
	=&U_j\left(T^{1,h}_t U_h\right) f=U_j\left(U_hT^1_t\right)f=\left(U_jU_h\right)T^1_tf. 
\end{aligned}\]
Finally $U_hU_j$ is unitary because so are $U_h$ and $U_j$. 

\emph{Necessity}.  The idea of this proof is due to \cite{LSW18} and we present some details for readers' convenience.  In view of Proposition~\ref{PRO26},  let $s$ and $\tau$ be the scaling and transformation associated with $U\in \mathscr U$.  Recall that $T^i_t$ and $U$ also stand for the extended operators on $L_+(E_1,m_1)$ or $L_+(E_2,m_2)$.  Particularly \eqref{eq:24} and \eqref{eq:31} also hold for $f\in L_+(E_1,m_1)$,  and the Markovian property of $(\sE^i,\sF^i)$ implies that $T^i_t 1\leq 1$.  Set $h:=s$.  
 It follows from
\[
	UT^1_t h =T^2_t Uh=T^2_t 1\leq 1
\]
that $T^1_t h\leq U^{-1}1=h$.  Then \cite[Lemma~3.4 and Proposition~3.7]{LSW18} indicate that $h$ admits an $\sE^1$-q.c.  $m_1$-version $\tilde{h}$.  Repeating the argument in the proof of \cite[Lemma~3.9]{LSW18},  we can obtain that $\{\tilde{h}=0\}$ is $\sE^1$-polar,  and thus $h\in \mathbf{Exc}^+_1$.  In addition,  the same argument as the proof of \cite[Theorem~3.11]{LSW18} shows that $\tau$ admits an $m_1$-version $\tilde{\tau}$ and there exist $\sE^i$-nests $\{F^i_n\}$ such that $\tilde{\tau}|_{F^1_n}:F^1_n\rightarrow F^2_n$ is a topological homeomorphism.  Set $j:=\tilde{\tau}$.  We show that $j$ is a quasi-homeomorphism from $(\sE^{1,h},\sF^{1,h})$ to $(\sE^2,\sF^2)$.  In fact,  the condition (i) in the definition of quasi-homeomorphism has been obtained.  To conclude (ii),  it follows from \eqref{eq:24} and the unitary property of $U$ that for any $f,g\in L^2(E_1,h^2\cdot m_1)$,  
\[
	\int_{E_1}fg h^2dm_1=\int_{E_2}U(fh)U(gh)dm_2=\int_{E_1} fgd(m_2\circ j).  
\]
Hence $m_2=(h^2\cdot m_1)\circ j^{-1}$ and (ii) holds true.  For (iii),  \eqref{eq:32} tells us that $f\in \sF^2$ if and only if $U^{-1}f\in \sF^1$.  Note that $U^{-1}f=h\cdot f\circ j\in \sF^1$,  if and only if $f\circ j\in \sF^{1,h}$.  Consequently $f\in \sF^2$ amounts to $f\circ j\in \sF^{1,h}$.  Meanwhile,  \eqref{eq:21} and \eqref{eq:32} yield that for $f,g\in\sF^2$,
\[
	\sE^2(f,g)=\sE^1(U^{-1}f,U^{-1}g)=\sE^{1,h}(f\circ j, g\circ j).  
\]
As a result (iii) is obtained.  Finally it suffices to note that, on account of \eqref{eq:25},  \eqref{eq:26} and \eqref{eq:24},  $U=U_jU_h$ holds true.  

\emph{Uniqueness}.  Let $(\tilde{h}_1, j_1)$ be another pair for $U$ such that $U=U_{j_1}U_{h_1}$.   Then $h=h_1$ and $j=j_1$,  $m_1$-a.e.,  due to the uniqueness stated in Proposition~\ref{PRO26}.  The identity $j_1=j$,  $\sE^1$-q.e.,  is a result of \cite[Lemma~3.10]{LSW18}.  Another identity $\tilde{h}=\tilde{h}_1$, $\sE^1$-q.e.,  follows because both $\tilde{h}$ and $\tilde{h}_1$ are $\sE^1$-q.c.  That completes the proof.  
\end{proof}

This characterization particularly shows that the existence of unitary order isomorphism intertwining $(T^1_t)_{t\geq 0}$ and $(T^2_t)_{t\geq 0}$,  which depends only on the measurable structures of $E_1$ and $E_2$,  leads to the coincidence of topological structures of $E_1$ and $E_2$ in a q.e.  sense.  That is the following.

\begin{corollary}\label{COR32}
Adopt the assumptions of Theorem~\ref{THM31}.  If $\mathscr U\neq \emptyset$,  then there exist $\sE^i$-nests $\{F_n^i:n\geq 1\}$ for $i=1,2$ and a map $j: \cup_{n\geq 1}F^1_n\rightarrow \cup_{n\geq 1}F^2_n$ such that $j|_{F^1_n}$ is a homeomorphism from $F^1_n$ to $F^2_n$ for any $n\geq 1$. 
\end{corollary}


\section{General order isomorphism intertwining semigroups}\label{SEC4}

\subsection{Invariant sets}
Let $(\sE,\sF)$ be a Dirichlet form on $L^2(E,m)$ and $(T_t)_{t\geq 0}$ be its $L^2$-semigroup.   
A set $A\in \mathcal{B}^{m}(E)$ is called \emph{$T_t$-invariant} if 
\[
	1_A \cdot T_t f=T_t\left(1_A f\right),\quad f\in L^2(E,m).  
\]
 Given a $T_t$-invariant set $A$, we can define the restriction of $(\sE,\sF)$ to $A$ as follows:
\[
\sF^{A}:=\{f|_A:f\in \sF\},\quad \sE^{A}(f|_A,g|_A):=\sE(1_Af,1_Ag),\; f,g\in\sF.
\]
Note that $(\sE^{A},\sF^{A})$ is a Dirichlet form on $L^2(A,  m|_A)$ whose $L^2$-semigroup is 
\[
	T_t^{A} \left(f|_{A}\right)= T_t(f 1_{A})|_{A},\quad f\in L^2(E,m).
\] 
We would write $m$ for $m|_A$ if no confusions caused. 
For $f\in L^2(A,m)$,  denote its zero extension to $E$ still by $f$ if there is no risk of ambiguity.  In view of \cite[Proposition~2.1.6]{CF12},  we may write $L^2(A,m)\subset L^2(E,m)$ and $\sF^A\subset \sF$ in abuse of notations.  

\begin{lemma}\label{LM42}
Assume that $E$ is a Hausdorff topological space and let $A$ be a $T_t$-invariant set.  Endow $A$ with the restricted topology of $E$.  Then the following hold:
\begin{itemize}
\item[(1)] If $\{F_n: n\geq 1\}$ is an $\sE$-nest,  then $\{F_n\cap A: n\geq 1\}$ is an $\sE^A$-nest. 
\item[(2)] $f|_A$ is $\sE^A$-q.c.  for any $\sE$-q.c. function $f$.  
\item[(3)] If $N\subset E$ is $\sE$-polar,  then $N\cap A$ is $\sE^A$-polar.  
\end{itemize}
\end{lemma}
\begin{proof}
We only need to prove the first assertion.  Set $\hat{F}_n:=F_n\cap A$. It suffices to show $\cup_{n\geq 1}\sF^A_{\hat{F}_n}$ is $\sE^A_1$-dense in $\sF^A$.  To accomplish this,  take $f\in \sF^A\subset \sF$.  Then there exists $\{g_n\}\subset \cup_{n\geq 1}\sF_{F_n}$ such that $\|g_n-f\|_{\sE_1}\rightarrow 0$.  Without loss of generality assume that $g_n\in \sF_{F_n}$.  Set $f_n:=g_n|_A\in \sF^A$.  Since $g_n=0$,  $m$-a.e. on $E\setminus F_n$,  it follows that $f_n=0$,  $m$-a.e.  on $A\setminus \hat{F}_n$.  Hence $f_n\in \sF^A_{\hat{F}_n}$.  Applying \cite[Proposition~2.1.6]{CF12},  we get that
\[
	\|f_n-f\|_{\sE^A_1}\leq \|g_n-f\|_{\sE_1}\rightarrow 0.  
\]
Therefore the assertion can be concluded. 
\end{proof}

In addition,  $(T_t)_{t\geq 0}$ or $(\sE,\sF)$ is called \emph{irreducible},  if either $m(A)=0$ or $m(E\setminus A)=0$ for any $T_t$-invariant set $A$.  It is worth pointing out that the definitions of invariant set and irreducibility depend only on the measurable structure of $E$.

\subsection{Invariant sets under order isomorphism}

From now on let $E_i$ be a Polish topological space and $(\sE^i,\sF^i)$ be a quasi-regular Dirichlet form on $L^2(E_i,m_i)$ whose $L^2$-semigroup is denoted by $(T^i_t)_{t\geq 0}$ for $i=1,2$.  Further let $U\in \mathscr O$,  whose scaling and transformation are $s$ and $\tau$ respectively. 

\begin{lemma}\label{LM41} 
$A$ is $T^1_t$-invariant,  if and only if $\tau(A)$ is $T^2_t$-invariant.  Particularly,  $(\sE^1,\sF^1)$ is irreducible,  if and only if so is $(\sE^2,\sF^2)$.  
\end{lemma}
\begin{proof}
Note that $U:L^2(E_1,m_1)\rightarrow L^2(E_2,m_2)$ is a bijective and write $g:=Uf$ for $f\in L^2(E_1,m_1)$.  Then $f=U^{-1}g$. We have
\[
	1_{\tau(A)} T_t^2 g=1_{\tau(A)}UT_t^1f=1_{\tau(A)}\cdot \frac{T^1_t f}{s}\circ \tau^{-1}=\frac{1_A T^1_t f}{s}\circ \tau^{-1}=U\left(1_A T^1_tf \right)
\]
and 
\[
	T^2_t(1_{\tau(A)}g)=T^2_t\left(1_{\tau(A)}Uf\right)=T^2_t\left(\frac{1_A f}{s}\circ \tau^{-1}\right)=T^2_tU(1_A f)=U\left(T^1_t (1_Af)\right).
\]
Hence $1_{\tau(A)} T_t^2 g=T^2_t(1_{\tau(A)}g)$ amounts to $1_A T^1_tf=T^1_t (1_Af)$.  In other words,  $A$ is $T^1_t$-invariant,  if and only if $\tau(A)$ is $T^2_t$-invariant.  

For the equivalence of irreducibility of $(\sE^1,\sF^1)$ and $(\sE^2,\sF^2)$,  it suffices to prove that $m_1\circ \tau^{-1}$ and $m_2$ are mutually absolutely continuous.  In fact,  take $A\in \mathcal{B}^{m_2}(E_2)$ with $m_2(A_2)=0$.  Then $1_A=0$ in $L^2(E_2,m_2)$ and hence $U^{-1}1_A=0$ in $L^2(E_1,m_1)$.  It follows from 
\[
	m_1\circ \tau^{-1}(A)=\int_{E_1} 1_A\circ \tau dm_1=\int_{E_1}\frac{U^{-1}1_A}{s} dm_1=0
\]
that $m_1\circ \tau^{-1}\ll m_2$.  The contrary can be obtained analogically.  That completes the proof. 
\end{proof}

Denote by $\|U\|$ the operator norm of $U\in \mathscr O$.  The following result characterizes all order isomorphisms intertwining semigroups for irreducible case; see also \cite{LSW18}.

\begin{proposition}\label{PRO42}
Adopt the assumptions of Theorem~\ref{THM31} and assume further that $(\sE^1,\sF^1)$ is irreducible.  Then $U\in \mathscr O$,  if and only if there exists $h\in \mathbf{Exc}^+_1$ and a quasi-homeomorphism $j$ from $(\sE^{1,h},\sF^{1,h})$ to $(\sE^2,\sF^2)$ such that $U=c\cdot  U_j U_h$ for some constant $c>0$. 
The pair $(\tilde{h},j)$ for $U\in \mathscr O$ is unique up to $\sE^1$-q.e.  equality. 
\end{proposition}
\begin{proof}
The sufficiency is clear because
\[
	U_c: L^2(E_2,m_2)\rightarrow L^2(E_2,m_2),\quad f\mapsto c\cdot f
\]
is an order isomorphism such that $U_cT^2_t=T^2_tU_c$.  For the necessity,  on account of Lemma~\ref{LM41},  $(\sE^2,\sF^2)$ is irreducible.  Then \cite[Corollary~2.4]{LSW18} leads to that $U/\|U\|$ is unitary for $U\in \mathscr O$.  In other words,  $U/\|U\|\in \mathscr U$.  As a result,  the representation of $U$ with $c=\|U\|$ as well as the uniqueness of $(\tilde h, j)$ is a consequence of Theorem~\ref{THM31}.  That completes the proof.  
\end{proof}

Obviously the analogical result of Corollary~\ref{COR32} holds.  

\begin{corollary}
Adopt the same assumptions of Proposition~\ref{PRO42}.  If $\mathscr O\neq \emptyset$,  then there exist $\sE^i$-nests $\{F_n^i: n\geq 1\}$ for $i=1,2$ and a map $j: \cup_{n\geq 1}F^1_n\rightarrow \cup_{n\geq 1}F^2_n$ such that $j|_{F^1_n}$ is a homeomorphism from $F^1_n$ to $F^2_n$ for any $n\geq 1$. 
\end{corollary}

\subsection{General characterization}
Take a $T^1_t$-invariant set $A$.  
Denote by $T^{1,A}_t$ and $T^{2,\tau(A)}_t$ the restrictions of $T^1_t$ to $A$ and $T^2_t$ to $\tau(A)$ respectively.  Clearly \eqref{eq:24} yields that $(A,\tau(A))$ is $U$-invariant in the sense that
\[
	U(f1_A)=1_{\tau(A)}Uf,\quad f\in L^2(E_1,m_1).  
\]
As a consequence,  
\begin{equation}\label{eq:41}
	U^A(f|_A):=U(f1_A)|_{\tau(A)}\in L^2(\tau(A),  m_2),\quad f\in L^2(E_1,m_1)
\end{equation}
is well defined. 

\begin{lemma}\label{LM44}
Adopt the same assumptions of Lemma~\ref{LM41} and let $A$ be a $T^1_t$-invariant set.  Then \eqref{eq:41} gives an order isomorphism $$U^A: L^2(A,m_1)\rightarrow L^2(\tau(A), m_2)$$ intertwining $T^{1,A}_t$ and $T^{2,\tau(A)}_t$.  
\end{lemma}
\begin{proof}
Clearly $U^A$ is linear and  injective.  For $g\in L^2(\tau(A),m_2)\subset L^2(E_2,m_2)$,  set $f:=(U^{-1}g)|_A\in L^2(A,m_1)$.  Then
\[
	U^A f = U(U^{-1}g \cdot 1_A)|_{\tau(A)}=U(U^{-1}g)|_{\tau(A)}=g.  
\]  
Hence $U^A$ is surjective.  Note that $f\geq 0$ if and only if $U^A f\geq 0$ for any $f\in L^2(A,m_1)$.  Thus $U^A$ is an order isomorphism.  Since $\tau(A)$ is $T^2_t$-invariant, it follows that for any $f\in L^2(A,m_1)\subset L^2(E_1,m_1)$, 
\[
\begin{aligned}
	&U^AT^{1,A}_tf=U^A\left(T^1_tf|_A \right)=U(T^1_tf)|_{\tau(A)}=T^2_t (Uf)|_{\tau(A)}\\=&T^{2}_t(Uf\cdot 1_{\tau(A)})|_{\tau(A)}
	=T^{2,\tau(A)}_t(Uf|_{\tau(A)})=T^{2,\tau(A)}_t(Uf|_{\tau(A)})=T^{2,\tau(A)}_t U^A f.  
\end{aligned}\] 
That completes the proof. 
\end{proof}

Denote by $X^1=(X^1_t)_{t\geq 0}$ the Markov process associated with $(\sE^1,\sF^1)$ and by $(P^1_t(x,dy))_{t\geq 0}$ the probability transition semigroup of $X^1$.  We impose the following assumption:
\begin{description}
\item[(AC)] $(P^1_t)_{t\geq 0}$ satisfies the absolute continuity condition with respect to $m_1$,  i.e.  $P^1_t(x,dy)\ll m_1(dy)$ for any $t\geq 0$ and $x\in E_1$.  
\end{description}
Under \textbf{(AC)},  Kuwae \cite{K21} shows that $E_1$ admits a unique irreducible decomposition for $(\sE^1,\sF^1)$:
\begin{equation}\label{eq:44}
	E_1=\cup_{1\leq n\leq N} A_n^1,
\end{equation}
where $\{A_n^1\in \mathcal{B}^{m_1}(E_1): 1\leq n\leq N\}$,  $N\in \mathbb N$ or $\infty$,  forms a disjoint union,  and each $A_n^1$ is $T^1_t$-invariant and contains no proper $T^1_t$-invariant subsets.  Particularly the restriction of $(\sE^1,\sF^1)$ to $A_n^1$ is irreducible.  Set $A_n^2:=\tau(A_n^1)$.  On account of Lemma~\ref{LM41}, 
\begin{equation}\label{eq:42}
	E_2=\cup_{1\leq n\leq N} A_n^2,\quad m_2\text{-a.e.}
\end{equation}
forms an irreducible decomposition for $(\sE^2,\sF^2)$.  
Define a family of step functions on $E_2$ as follows:
\begin{equation}\label{eq:43}
	\mathscr S:=\left\{\varphi=\sum_{n=1}^N c_n\cdot 1_{A_n^2}: \exists c>1,\text{ s.t. } 1/c\leq c_n\leq c,  1\leq n\leq N\right\}.  
\end{equation}
For any $\varphi\in \mathscr S$,  define
\begin{equation}\label{eq:45}
	U_\varphi: L^2(E_2,m_2)\rightarrow L^2(E_2,m_2),\quad f\mapsto \varphi\cdot f.
\end{equation}
Clearly $U_\varphi$ is bijective and $f\geq 0$ if and only if $U_\varphi f\geq 0$.  Consequently,  $U_\varphi$ is an order isomorphism.  
The main result is the following.

\begin{theorem}\label{THM45}
Adopt the assumptions of Theorem~\ref{THM31} and assume that $\mathbf{(AC)}$ holds.  Then $U\in \mathscr O$,  if and only if
\begin{itemize}
\item[(i)] $E_2$ admits an irreducible decomposition \eqref{eq:42} for $(\sE^2,\sF^2)$;
\item[(ii)]  There exists $h\in \mathbf{Exc}^+_1$,  a quasi-homeomorphism $j$ from $(\sE^{1,h},\sF^{1,h})$ to $(\sE^2,\sF^2)$ and $\varphi\in \mathscr S$ such that 
\begin{equation}\label{eq:46}
U=U_\varphi U_jU_h. 
\end{equation}
\end{itemize}
The triple $(h,j, \varphi)$ for $U\in \mathscr O$ is unique in the sense that if $(h_1,j_1,\varphi_1)$ is another triple such that $U=U_{\varphi_1} U_{j_1}U_{h_1}$,  then 
\[
\begin{aligned}
&\varphi=\varphi_1,\quad m_2\text{-a.e.},  \\
&j=j_1, \quad \tilde{h}=\tilde{h}_1,\quad \sE^1\text{-q.e.}
\end{aligned}\]
\end{theorem}
\begin{proof}
\emph{Sufficiency}.  Applying Theorem~\ref{THM31},  we only need to show $U_\varphi T^2_t=T^2_tU_\varphi$.  In fact,  let $\varphi=\sum_{n=1}^Nc_n\cdot 1_{A_n^2}$ as in \eqref{eq:43}.  Since \eqref{eq:42} forms an irreducible decomposition for $(\sE^2,\sF^2)$,  it follows that for any $g\in L^2(E_2,m_2)$,
\[
	1_{A_n^2}\cdot T^2_t g=T^2_t(g\cdot 1_{A_n^2}). 
\]
Hence 
\begin{equation}\label{eq:47}
U_\varphi T^2_t g=\sum_{n=1}^N c_n 1_{A_n^2}\cdot T^2_t g=\sum_{n=1}^N T^2_t(c_n g 1_{A_n^2})=T^2_t\left(\sum_{n=1}^N c_n g 1_{A_n^2} \right)=T^2_tU_\varphi g,
\end{equation}
where the third identity is due to the convergence of $\sum_{n=1}^N c_n g 1_{A_n^2}$ in $L^2(E_2,m_2)$.  

\emph{Necessity}.  Let \eqref{eq:44} be the irreducible decomposition for $(\sE^1,\sF^1)$.  Set $A^2_n:=\tau(A^1_n)$.  Then \eqref{eq:42} is an irreducible decomposition for $(\sE^2,\sF^2)$,  i.e.  the first assertion holds true.  In view of Lemma~\ref{LM44},  $$U_n:=U^{A^1_n}: L^2(A^1_n,m_1)\rightarrow L^2(A^2_n, m_2)$$ is an order isomorphism intertwining $T^{1,n}_t:=T^{1,A^1_n}_t$ and $T^{2, n}_t:=T^{2,A^2_n}_t$.  Denote by $c_n:=\|U_n\|$ the operator norm of $U_n$,  and set
\[
	\varphi:= \sum_{n=1}^N c_n\cdot 1_{A^2_n}.  
\]
We first prove that $\varphi\in \mathscr S$.  In fact,  for any $f\in L^2(A_1,m_1)\subset L^2(E_1,m_1)$, 
\[
\begin{aligned}
	\|U_n f\|_{L^2(A^2_n,m_2)}&=\|(Uf)|_{A^2_n}\|_{L^2(A^2_n,m_2)}\leq \|Uf\|_{L^2(E_2,m_2)} \\
	&\leq \|U\|\cdot \|f\|_{L^2(E_2,m_2)}=\|U\|\cdot \|f\|_{L^2(A^2_n,m_2)}.
	\end{aligned}\]
 This implies $\|U_n\|\leq \|U\|$.   Analogically we have $\|U^{-1}_n\|\leq \|U^{-1}\|$.  In addition,  it follows from $\mathbf{Id}^1_n=U^{-1}_nU_n$,  where $\mathbf{Id}^1_n$ is the identity operator on $L^2(A^1_n,m_1)$,  that 
	\[
		c_n=\|U_n\|\geq 1/ \|U^{-1}_n\|\geq 1/\|U^{-1}\|.  
	\]
As a result,  $\varphi\in \mathscr S$ can be concluded.  
Secondly we assert that $U_n/\|U_n\|$ is unitary.  Denote by $U^*_n$ the adjoint operator of $U_n$.  Set $s_n:=s|_{A^1_n}$ and $\tau_n:=\tau|_{A^1_n}$.   For any $f,g\in L^2(A^1_n, m_1)$,  we have
\begin{equation}\label{eq:49}
	(U^*_nU_nf,g)_{L^2(A^1_n,m_1)}=(U_nf,U_ng)_{L^2(A^2_n,m_2)}=\int_{A^1_n}\frac{ fg}{s_n^2} dm_2\circ \tau_n.
\end{equation}
In view of the proof of Lemma~\ref{LM41},  we have $m_2\circ \tau_n\ll m_1$.  Set $\rho_n:=dm_2\circ \tau_n/dm_1$.  It follows from \eqref{eq:49} that $U^*_nU_n f=\phi_n f$ where $\phi_n:=\rho_n/s_n^2$.  On the other hand,  $U_nT^{1,n}_t=T^{2,n}_tU_n$ implies $T^{1,n}_t U^*_n=U^*_nT^{2,n}_t$.  Thus
\[
\phi_n T^{1,n}_t f=U^*_nU_n T^{1,n}_tf=T^{1,n}_t U^*_n U_n f=T^{1,n}_t(\phi_n f).  
\]
Mimicking the proof of \cite[Lemma~2.2]{LSW18},  we can obtain that $\phi_n$ is constant,  $m_1$-a.e. on $A^1_n$.  Then \eqref{eq:49} leads to 
\[
(U_nf,U_ng)_{L^2(A^2_n,m_2)}=c^2\cdot (f,g)_{L^2(A^1_n, m_1)}
\]
for some constant $c>0$.  Therefore $c=\|U_n\|$ and $U_n/\|U_n\|$ is unitary. 
Finally,  let $U^{-1}_\varphi$ be the inverse of \eqref{eq:45},  and set $\tilde{U}:=U^{-1}_\varphi U$.  We show $\tilde{U}\in \mathscr U$,  so that Theorem~\ref{THM31} yields the representation \eqref{eq:46}.  To accomplish this,  note that $\tilde{U}$ is clearly bijective and for $f\in L^2(E_1,m_1)$,  $f\geq 0$ if and only if $\tilde{U}f\geq 0$.  Thus $\tilde{U}$ is an order isomorphism.  Mimicking \eqref{eq:47}, we get that 
\[
	U^{-1}_\varphi T^2_tg=T^2_t U^{-1}_\varphi g,\quad g\in L^2(E_2,m_2). 
\]
It follows from $U\in \mathscr O$ that for any $f\in L^2(E_1,m_1)$, 
\[
\tilde{U}T^1_t f=U^{-1}_\varphi T^2_t(Uf)=T^2_t U^{-1}_\varphi(Uf)=T^2_t\tilde{U}f. 
\]
Consequently $\tilde{U}\in \mathscr O$.  In addition,  \eqref{eq:41} yields that
\begin{equation}\label{eq:48}
\begin{aligned}
	(\tilde{U}f,\tilde{U}f)_{L^2(E_2,m_2)}&=\sum_{n=1}^N c^2_n\cdot \left(U f|_{A^2_n},  U f|_{A^2_n}\right)_{L^2(E_2,m_2)}\\
	&=\sum_{n=1}^N c^2_n\cdot \left(U_n (f|_{A^1_n}),  U_n (f|_{A^1_n})\right)_{L^2(A_n^2,m_2)}.  
\end{aligned}\end{equation}
Since $U_n/\|U_n\|$ is unitary,  it follows that
\[
(\tilde{U}f,\tilde{U}f)_{L^2(E_2,m_2)}=\sum_{n=1}^N(f|_{A^1_n},f|_{A^1_n})_{L^2(A^1_n,m_1)}=(f,f)_{L^2(E_1,m_1)}.
\]
Therefore $\tilde{U}\in \mathscr U$.  

\emph{Uniqueness}.   Note that $U^{-1}_{\varphi_1}U=U_{j_1}U_{h_1}$ is unitary.  Mimicking \eqref{eq:48},  we can obtain that $\varphi_1=\varphi$,  $m_2$-a.e.  Particularly,  $U_\varphi=U_{\varphi_1}$ and 
\[
	U_{j_1}U_{h_1}=U^{-1}_{\varphi_1} U=U_\varphi^{-1}U=U_jU_h.  
\]
The second identity is a consequence of the uniqueness obtained in Theorem~\ref{THM31}.  That completes the proof. 
\end{proof}

We end this section with two corollaries of Theorem~\ref{THM45}.  The first one is the analogue of Corollary~\ref{COR32},  and the proof is trivial. 

\begin{corollary}
Adopt the same assumptions of Theorem~\ref{THM45}.  If $\mathscr O\neq \emptyset$,  then there exist $\sE^i$-nests $\{F_n^i: n\geq 1\}$ for $i=1,2$ and a map $j: \cup_{n\geq 1}F^1_n\rightarrow \cup_{n\geq 1}F^2_n$ such that $j|_{F^1_n}$ is a homeomorphism from $F^1_n$ to $F^2_n$ for any $n\geq 1$. 
\end{corollary}

Let $U=U_\varphi U_j U_h$ be in Theorem~\ref{THM45},  and denote by $U_n:=U^{A^1_n}$ the restriction of $U$ to the $T^1_t$-invariant set $A^1_n$.  Although we cannot apply Proposition~\ref{PRO42} directly to $U_n$ because $A^1_n$ is not necessarily Polish,  the analogical representation holds true for $U_n$.  To accomplish this,  denote by $\mathbf{Exc}^+_{1,n}$ the family of excessive functions defined as \eqref{eq:21} with respect to $T^{1,A^1_n}_t$.  The $h$-transform of $(\sE^{1,A^1_n},\sF^{1,A^1_n})$ for $h_n\in \mathbf{Exc}^+_{1,n}$ is denoted by $(\sE^{1,A^1_n, h_n},\sF^{1,A^1_n,h_n})$.   Set
\[
	\tilde{A}^2_n:=j(A^1_n),\quad 1\leq n\leq N.  
\]
Then $\tilde{A}^2_n=A^2_n$,  $m_2$-a.e.  and 
\[
	E_2=\cup_{n=1}^N \tilde{A}^2_n,\quad m_2\text{-a.e.}
\]
also forms an irreducible decomposition for $(\sE^2,\sF^2)$.  The restriction of $(\sE^2,\sF^2)$ to $\tilde{A}^2_n$ is denoted by $(\sE^{2,\tilde{A}^2_n},\sF^{2, \tilde A^2_n})$

\begin{corollary}
Adopt the same assumptions of Theorem~\ref{THM45} and let $U=U_\varphi U_jU_h\in \mathscr O$.  For any $1\leq n\leq N$,  there exists $h_n\in \mathbf{Exc}^+_{1,n}$ and a quasi-homeomorphism $j_n$ from $(\sE^{1,A^1_n, h_n},\sF^{1,A^1_n,h_n})$ to $(\sE^{2,\tilde A^2_n},\sF^{2,\tilde A^2_n})$ such that $U_n=\|U_n\| U_{j_n}U_{h_n}$.  
\end{corollary}
\begin{proof}
Set $h_n:=h|_{A^1_n}$ and $j_n:=j|_{A^1_n}$.  It suffices to note that $h_n\in \mathbf{Exc}^+_{1,n}$ and $j_n$ is a quasi-homeomorphism from $(\sE^{1,A^1_n, h_n},\sF^{1,A^1_n,h_n})$ to $(\sE^{2,\tilde A^2_n},\sF^{2,\tilde A^2_n})$ by virtue of Lemma~\ref{LM42}. 
\end{proof}

\bibliographystyle{siam} 
\bibliography{OrdIsom} 

\begin{thebibliography}{1}

\bibitem{Aliprantis:2006jm}
{\sc C.~D. Aliprantis and O.~Burkinshaw}, {\em {Positive operators}}, Springer,
  Dordrecht, 2006.

\bibitem{A02}
{\sc W.~Arendt}, {\em {Does diffusion determine the body?}}, J. Reine Angew.
  Math., 550 (2002), pp.~97--123.

\bibitem{CF12}
{\sc Z.-Q. Chen and M.~Fukushima}, {\em {Symmetric Markov processes, time
  change, and boundary theory}}, vol.~35 of London Mathematical Society
  Monographs Series, Princeton University Press, Princeton, NJ, 2012.

\bibitem{DW21}
{\sc L.~Dello~Schiavo and M.~Wirth}, {\em {Ergodic decompositions of Dirichlet
  forms under order isomorphisms}}, arXiv: 2109.00615,  (2021).

\bibitem{FOT11}
{\sc M.~Fukushima, Y.~Oshima, and M.~Takeda}, {\em {Dirichlet forms and
  symmetric Markov processes}}, vol.~19 of de Gruyter Studies in Mathematics,
  Walter de Gruyter {\&} Co., Berlin, extended~ed., 2011.

\bibitem{GWW92}
{\sc C.~Gordon, D.~Webb, and S.~Wolpert}, {\em {Isospectral plane domains and
  surfaces via Riemannian orbifolds}}, Invent. Math., 110 (1992), pp.~1--22.

\bibitem{K66}
{\sc M.~Kac}, {\em {Can one hear the shape of a drum?}}, Amer. Math. Monthly,
  73 (1966), pp.~1--23.

\bibitem{K21}
{\sc K.~Kuwae}, {\em {Irreducible decomposition for Markov processes}},
  Stochastic Process. Appl., 140 (2021), pp.~339--356.

\bibitem{LSW18}
{\sc D.~Lenz, M.~Schmidt, and M.~Wirth}, {\em {Geometric properties of
  Dirichlet forms under order isomorphisms}}, arXiv: 1801.08326,  (2018).

\end{thebibliography}

\end{document}